\documentclass[11pt]{article}
\usepackage[a4paper,centering,vscale=0.75,hscale=0.7]{geometry}

\usepackage{mathtools}
\usepackage{amsthm,amssymb}
\usepackage{mathrsfs}

\newtheorem{thm}{Theorem}[section]
\newtheorem{lemma}[thm]{Lemma}
\newtheorem{prop}[thm]{Proposition}
\newtheorem{coroll}[thm]{Corollary}
\newtheorem{defi}[thm]{Definition}
\theoremstyle{remark}
\newtheorem{rmk}[thm]{Remark}

\newcommand{\E}{\mathop{{}\mathbb{E}}}
\newcommand{\cF}{\mathscr{F}}
\newcommand{\cL}{\mathscr{L}}
\renewcommand{\P}{\mathbb{P}}
\newcommand{\erre}{\mathbb{R}}
\newcommand{\enne}{\mathbb{N}}

\newcommand{\longto}{\longrightarrow}
\renewcommand{\div}{\mathop{\mathrm{div}}}

\DeclarePairedDelimiter\abs{\lvert}{\rvert}
\DeclarePairedDelimiter\norm{\lVert}{\rVert}
\DeclarePairedDelimiterX\ip[2]{\langle}{\rangle}{#1,#2}

\numberwithin{equation}{section}

\title{A note on doubly nonlinear SPDEs\\ with singular drift in
  divergence form} 

\author{
  Carlo Marinelli and Luca Scarpa\\[4pt]
  {\footnotesize \emph{Department of Mathematics, University College
      London, United Kingdom}} }

\date{\normalsize December 15, 2017}

\begin{document}
\maketitle

\begin{abstract}
  We prove well-posedness for a class of second-order SPDEs with
  multiplicative Wiener noise and doubly nonlinear drift of the form
  $-\div \gamma(\nabla \cdot) + \beta(\cdot)$, where $\gamma$ is the
  subdifferential of a convex function on $\erre^d$ and $\beta$ is a
  maximal monotone graph everywhere defined on $\erre$, on which
  neither growth nor continuity assumptions are imposed.
\end{abstract}

\section{Introduction}
Let $D$ be a bounded domain of $\erre^d$ with smooth boundary and
$T>0$ a fixed number. We shall establish well-posedness in the strong
sense for stochastic partial differential equations of the type
\begin{equation}
  \label{eq:0}
  \begin{cases}
    du(t) - \div\gamma(\nabla u(t))\,dt + \beta(u(t))\,dt \ni B(t,u(t))\,dW(t)
    \quad &\text{in } (0,T)\times D,\\
  u=0 \quad &\text{in } (0,T)\times\partial D,\\
  u(0)=u_0 \quad &\text{in } D,
  \end{cases}
\end{equation}
where $\gamma \subset \erre^d \times \erre^d$ and
$\beta \subset \erre \times \erre$ are everywhere-defined maximal
monotone graphs, the first one of which is assumed to be the
subdifferential of a convex function $k:\erre^d \to
\erre$. Furthermore, $W$ is a cylindrical Wiener process on a
separable Hilbert space $U$, and $B$ takes values in the space of
Hilbert-Schmidt operators from $U$ to $L^2(D)$.  Precise assumptions
on the data of the problem are given in {\S}\ref{sec:ass} below.

Equations with drift in divergence type, both in deterministic and
stochastic settings, have a long history and are thoroughly studied,
especially because of their physical significance. From a mathematical
point of view, they are particularly interesting because they are
fully nonlinear, in the sense that they do not contain any ``leading''
linear term. For stochastic equations, the first well-posedness result
is most likely due to Pardoux, as an application of his general
results in \cite{Pard} on monotone stochastic evolution equations in
the variational setting (see also \cite{KR-spde} for improved results
under more general assumptions on $B$). In this case one needs to
assume $\beta=0$ and
\[
  \gamma(x) \cdot x \gtrsim \abs{x}^p - 1,
  \qquad \abs{\gamma(x)} \lesssim \abs{x}^{p-1} - 1
  \qquad \forall x \in \erre^d,
\]
with $p>1$ (the centered dot stands for the usual Euclidean scalar
product in $\erre^d$). These are precisely the classical Leray-Lions
conditions, well known in the deterministic theory
(cf.~\cite{Lions:q}).  In some special cases a simple polynomial-type
$\beta$ can be added: for instance, if $\gamma$ corresponds to the
$p$-Laplacian, i.e.  $\gamma(x) = \abs{x}^{p-2}x$, $p \geq 2$, one may
consider $\beta(x)=\abs{x}^{p-2}x$ (cf.~\cite[p.~83]{LiuRo}). However,
it is well known that if two nonlinear operators satisfy the
conditions needed in the variational setting, their sum in general
does not. This phenomenon already gives rise to severe restrictions on
the class of semilinear equations with polynomial nonlinearities that
can be solved bu such methods.

In some recent works we have obtained well-posedness results for
\eqref{eq:0} under much more general hypotheses than those mentioned
above. In particular, in \cite{luca} we assume that $\gamma$ still
satisfies the classical Leray-Lions assumptions, but we impose no
growth restriction on $\beta$: a very mild symmetry-like condition on
its behavior at infinity is shown to suffice. On the other hand, in
\cite{cm:div} we consider the case $\beta=0$, with no hypotheses on
the growth of $\gamma$, but with the additional requirement that
$\gamma$ is single-valued (a symmetry-like assumption on $\gamma$ is
needed in this case as well). Equations with more general, possibly
multivalued $\gamma$, are treated in \cite{cm:div2}, where, however,
less regular solutions are obtained.

Our goal is to unify and extend the above-mentioned well-posedness
results for equation \eqref{eq:0}, thus treating the case where both
$\gamma$ and $\beta$ can be multivalued, without any restriction on
their rate of growth. We shall also show that we can do so without
loosing any regularity of solutions with respect to the results of
\cite{cm:div}. The approach we take, initiated in \cite{cm:luca} and
further refined and extended in \cite{cm:inv}--\cite{cm:div2},
consists in a combination of (deterministic and stochastic)
variational techniques and weak compactness in $L^1$ spaces. A key
feature is the construction of a candidate solution as pathwise limit,
in suitable topologies, of solutions to regularized equations. In
particular, due to this type of construction, in order to obtain
measurability properties of solutions, uniqueness of limits is
crucial. Roughly speaking, we can prove that
$-\div\gamma(\nabla u) + \beta(u)$ is unique, hence that it is
measurable, but showing that each one of them is unique (hence
measurable) seems difficult, if not impossible. This is the reason why
$\gamma$ was assumed to be single-valued in \cite{cm:div,luca}. In the
general setting of this work we thus need different ideas: let
$u_\lambda$, $\gamma_\lambda$, and $\beta_\lambda$ be suitable
regularizations of $u$, $\gamma$, and $\beta$, respectively, and set
$\eta_\lambda:=\gamma_\lambda(\nabla u_\lambda)$ and
$\xi_\lambda:=\beta_\lambda(u_\lambda)$. Comparing weak limits,
obtained in different ways, of the image of the pair
$(\eta_\lambda,\xi_\lambda)$ under a continuous linear map, we are
going to prove that there exist two limiting processes $\eta$ and
$\xi$, ``sections'' of $\gamma(\nabla u)$ and $\beta(u)$,
respectively, that are indeed predictable and satisfy suitable
uniqueness properties. One may say that we restore uniqueness working
in a suitable quotient space, although quotient spaces do not appear
explicitly.

The well-posedness result obtained here may be interesting also in the
deterministic setting, as our results extend to the doubly nonlinear
case the sharpest results available for equations with $\beta=0$ and
$B=0$, whose hypotheses on $\gamma$ are identical to ours
(cf.~\cite[p.~207-ff]{Barbu:type})

The paper is organized as follows: in Section~\ref{sec:ass} we state
the assumptions and the main result, which is then proved in
Section~\ref{sec:dim}.
\smallskip\par\noindent
\textbf{Acknowledgments.} Part of the work for this paper was done
while the authors were supported by a grant of the Royal Society. The
first-named author gratefully acknowledges the hospitality of the IZKS
at the University of Bonn.

\section{Main result}
\label{sec:ass}
Before stating the main result, we fix notation and introduce the
necessary assumptions.

As already mentioned, $D$ stands for a bounded domain in $\erre^d$
with smooth boundary. We shall denote the Hilbert space $L^2(D)$ by
$H$, its norm and scalar product by $\norm{\cdot}$ and
$\ip{\cdot}{\cdot}$, respectively.
We shall denote the Dirichlet Laplacian on $L^1(D)$ (as well as on
$L^2(D)$, without notationally distinguish them) by $\Delta$.
The space of Hilbert-Schmidt operators from the separable Hilbert
space $U$ to $H$ is denoted by $\cL^2(U,H)$.
We shall write $a \lesssim b$ to mean that there exists a constant $N>0$
such that $a \leq Nb$.

\smallskip

Let $(\Omega, \cF, \P)$ be a probability space, endowed with a
filtration $(\cF_t)_{t\in[0,T]}$ satisfying the so-called usual
conditions, on which all random elements will be defined. Equality of
stochastic processes is meant to be in the sense of
indistinguishability, unless otherwise stated.
We assume that the diffusion coefficient
\[
  B: \Omega \times [0,T] \times H \to \cL^2(U, H)
\]
is such that $B(\cdot,\cdot,h)$ is progressively measurable for all
$h \in H$, and there exists a positive constant $N_B$ such that
\begin{gather*}
  \norm[\big]{B(\omega,t,x)}_{\cL^2(U,H)} \leq N_B\bigl( 1+\norm{x} \bigr),\\
  \norm[\big]{B(\omega,t,x) - B(\omega,t,y)}_{\cL^2(U,H)} \leq N_B\norm{x-y}
\end{gather*}
for all $(\omega,t) \in \Omega \times [0,T]$ and $x,y \in
H$. Moreover, let the initial datum $u_0$ be $\cF_0$-measurable with
finite second moment, i.e.  $u_0 \in L^2(\Omega,\cF_0;H)$.

\smallskip

Let $k:\erre^d \to \erre_+$ be a convex function with $k(0)=0$ such that
\[
  \limsup_{|x|\to+\infty} \frac{k(x)}{k(-x)} < +\infty,
  \qquad
  \lim_{|x|\to +\infty} \frac{k(x)}{|x|} = +\infty
\]
(we shall call the second condition superlinearity at infinity). Then
its subdifferential $\gamma:=\partial k$ is a maximal monotone graph
in $\erre^d \times \erre^d$. We assume that the domain of $\gamma$
coincides with $\erre^d$, which implies that $k^*$, the convex
conjugate of $k$, is superlinear at infinity as well.
Moreover, let $j:\erre \to \erre_+$ be a further convex function with
$j(0)=0$ such that
\[
  \limsup_{|x|\to+\infty}\frac{j(x)}{j(-x)} < +\infty,
\]
whose subdifferential $\beta:=\partial j$ is an everywhere defined
maximal monotone graph in $\erre \times \erre$, so that $j^*$ is
superlinear at infinity. All notions of convex analysis and from the
theory of maximal monotone operators used thus far and in the sequel
are standard and are treated in detail, for instance, in
\cite{Barbu:type}.

\medskip

We can now give the notion of solution to \eqref{eq:0} that we are
going to work with. Throughout the work, $V_0$ is a separable Hilbert
space continuously embedded in both $W^{1,\infty}(D)$ and $H^1_0(D)$:
for instance one can take, thanks to Sobolev embedding theorems,
$V_0:=H^k_0(D)$ for $k \in \enne$ sufficiently large.
Moreover, the divergence operator is defined as
\begin{align*}
  \div: L^1(D)^d &\longto V_0'\\
  f &\longmapsto \bigl[ g \mapsto - \ip{f}{\nabla g} \bigr],
\end{align*}
which is thus linear and bounded. In fact, for any
$f \in L^1(D)^d$ and $g \in V_0$,
\[
  \abs[\big]{\ip{f}{\nabla g}}
  \leq \norm[\big]{f}_{L^1(D)} \norm[\big]{g}_{W^{1,\infty}}
  \lesssim \norm[\big]{f}_{L^1(D)} \norm[\big]{g}_{V_0}
\]
because $V_0$ is continuously embedded in $W^{1,\infty}$.
\begin{defi}
  A strong solution to \eqref{eq:0} is a triplet $(u,\eta,\xi)$, where
  $u$, $\eta$, and $\xi$ are adapted processes taking values in
  $W^{1,1}_0(D) \cap H$, $L^1(D)^d$, and $L^1(D)$, respectively, such
  that $\eta \in \gamma(\nabla u)$ and $\xi \in \beta(u)$ a.e. in
  $\Omega \times (0,T) \times D$,
  \begin{align*}
    u &\in L^0(\Omega; C([0,T]; H)) \cap L^0(\Omega; L^1(0,T; W^{1,1}_0(D))),\\
    \eta &\in L^0(\Omega; L^1((0,T)\times D)^d),\\
    \xi &\in L^0(\Omega; L^1((0,T)\times D)),\\
    \nabla u \cdot \eta + u \xi &\in L^0(\Omega;L^1((0,T) \times D)),
  \end{align*}
  and
  \[
    \ip[\big]{u}{\phi} + \int_0^\cdot \ip[\big]{\eta(s)}{\nabla\phi}\,ds
    + \int_0^\cdot \ip[\big]{\xi(s)}{\phi}\,ds = \ip{u_0}{\phi}
    + \ip[\bigg]{\int_0^\cdot B(s,u(s))\,dW(s)}{\phi}
  \]
  for all $\phi \in V_0$.
\end{defi}
\noindent The last identity in the above definition is equivalent to the
validity in the dual of $V_0$ of the equality
\[
  u - \int_0^\cdot \div\eta(s)\,ds
  + \int_0^\cdot \xi(s)\,ds = u_0
  + \int_0^\cdot B(s,u(s))\,dW(s).
\]
Note that $u$, $u_0$ and the stochastic integrals take values in $H$
and the third term on the left-hand side takes values in $L^1(D)$,
hence also the second term on the right-hand side belongs to $L^1(D)$,
so that the equality holds also in $L^1(D)$. The same reasoning
implies that the sum of the second and third terms on the left-hand
side take values in $H$, so that the above equality can also be seen
as valid in $H$.

The main result of the paper is the following. The proof is given in
{\S}\ref{sec:dim} below.
\begin{thm}
  \label{thm:main}
  There exists a strong solution $(u,\eta,\xi)$ to equation
  \eqref{eq:0}. It is predictable and satisfies the following
  properties:
  \begin{align*}
    u &\in L^2(\Omega; C([0,T]; H))\cap L^1(\Omega; L^1(0,T; W^{1,1}_0(D))),\\
    \eta &\in L^1(\Omega\times(0,T)\times D)^d,\\
    \xi &\in L^1(\Omega\times(0,T)\times D),\\
    \nabla u \cdot \eta &\in L^1(\Omega\times(0,T)\times D),\\    
    u \xi &\in L^1(\Omega\times(0,T)\times D).
  \end{align*}
  Moreover, the solution map
  \begin{align*}
    L^2(\Omega,\cF_0; H) &\longto L^2(\Omega; C([0,T]; H))\\
    u_0 &\longmapsto u
  \end{align*}
  is Lipschitz-continuous. 
  In particular, if $(u_1,\eta_1, \xi_1)$ and $(u_2,\eta_2,\xi_2)$
  are any two strong solutions satisfying the properties above,
  then $u_1=u_2$ and $-\div\eta_1+\xi_1=-\div\eta_2+\xi_2$
  in $L^2(\Omega; C([0,T]; H))$ and $L^1(\Omega; L^1(0,T; V_0'))$,
  respectively.
\end{thm}

\section{Proof of Theorem~\ref{thm:main}}
\label{sec:dim}

\subsection{It\^o's formula for the square of the $H$-norm}
We establish a version of It\^o's formula for the square of the
$H$-norm in a generalized variational setting, which will play an
important role in the sequel. The result is interesting in its own
right, as it does not follow from the classical ones in
\cite{KR-spde,Pard}, and is apparently new for It\^o processes
containing a drift term in divergence form with minimal integrability
properties.
\begin{prop}\label{prop:Ito}
  Let $Y$, $f$, and $g$ be measurable adapted processes with values in
  $H \cap W^{1,1}_0(D)$, $L^1(D)^d$, and $L^1(D)$, respectively, such
  that
  \begin{align*}
    Y &\in L^0(\Omega;L^\infty(0,T;H)) \cap L^0(\Omega;L^1(0,T;W^{1,1}_0(D))),\\
    f &\in L^0(\Omega;L^1((0,T)\times D)^d),\\
    g &\in L^0(\Omega; L^1((0,T)\times D)),
  \end{align*}
  and there exists constants $a$, $b>0$ such that
  \[
    k(a\nabla u) + k^*(a f) + j(bu) + j^*(bg)
    \in L^0(\Omega;L^1((0,T) \times D)).
  \]
  Moreover, let $Y_0 \in L^0(\Omega,\cF_0; H)$ and $G$ be an
  $\cL^2(U,H)$-valued progressively measurable process such that
  $G \in L^0(\Omega;L^2(0,T;\cL^2(U,H)))$. If 
  \[
    Y - \int_0^\cdot \div f(s)\,ds + \int_0^\cdot g(s)\,ds
    = Y_0 + \int_0^\cdot G(s)\,dW(s)
  \]
  as an identity in $V_0'$, then
  \begin{align*}
    &\frac12\norm{Y}^2 + \int_0^\cdot\!\!\int_D f(s) \cdot \nabla Y(s)\,ds
      + \int_0^\cdot\!\!\int_D g(s)Y(s)\,ds \\
    &\hspace{3em} = \frac12\norm{Y_0}^2 +
      \frac12 \int_0^\cdot \norm{G(s)}^2_{\cL^2(U,H)}\,ds
      + \int_0^\cdot Y(s)G(s)\,dW(s).
  \end{align*}
\end{prop}
\begin{proof}
  The proof is essentially a combination of arguments described in
  great detail in \cite{cm:ref,cm:div}, hence we shall limit ourselves
  to a sketch only.
  Using a superscript $\delta$ to denote the action of
  $(I-\delta\Delta)^{-m}$, for a sufficiently large $m\in\enne$, we
  have, thanks to Sobolev embedding theorems and classical elliptic
  regularity results,
  \[
    Y^\delta - \int_0^\cdot \div f^\delta(s)\,ds + \int_0^\cdot g^\delta(s)\,ds
    = Y^\delta_0 + \int_0^\cdot G^\delta(s)\,dW(s)
  \]
  as an identity of $H$-valued processes. It\^o's formula for
  Hilbert-space valued continuous semimartingales thus yields
  \begin{equation}
    \label{eq:delta}
    \begin{split}
    &\frac12 \norm{Y^\delta}^2
      + \int_0^\cdot\!\!\int_D f^\delta(s) \cdot \nabla Y^\delta(s)\,ds
      + \int_0^\cdot \!\!\int_D g^\delta(s)Y^\delta(s)\,ds\\
    &\hspace{3em} = \frac12\norm{Y^\delta_0}^2
      + \frac12\int_0^\cdot \norm{G^\delta(s)}^2_{\cL^2(U,H)}\,ds
      + \int_0^\cdot Y^\delta(s)G^\delta(s)\,dW(s).
    \end{split}
  \end{equation}
  Thanks to the assumptions on $Y$, $f$, $g$ ad $G$, it easily follows
  that, $\P$-a.s.,
  \begin{alignat*}{2}
    Y_0^\delta &\longto Y_0 &&\quad \text{in } H,\\
    Y^\delta(t) &\longto Y(t) &&\quad \text{in } H
    \quad \forall\,t\in[0,T],\\
  f^\delta  &\longto f &&\quad \text{in } L^1((0,T)\times D)^d,\\
  g^\delta &\longto g &&\quad \text{in } L^1((0,T)\times D),\\
  G^\delta &\longto G &&\quad \text{in } L^2(0,T; \cL^2(U,H)).
  \end{alignat*}
  Similarly, using simple properties of Hilbert-Schmidt operators and
  the dominated convergence theorem, it is not difficult to verify
  that the quadratic variation of $(Y^\delta G^\delta- YG) \cdot W$
  converges to zero in probability, so that
  \[
    \int_0^\cdot Y^\delta G^\delta \,dW \longto
    \int_0^\cdot YG \,dW
  \]
  uniformly (with respect to time) in probability.  Furthermore,
  thanks to the hypotheses on $k$ and $j$, the families
  $(\nabla u^\delta\cdot Y^\delta)$ and $(g^\delta Y^\delta)$ are
  uniformly integrable in $(0,T)\times D$ $\P$-a.s., hence by Vitali's
  theorem we also have that, $\P$-a.s.,
  \begin{alignat*}{2}
    f^\delta \cdot \nabla Y^\delta &\longto f \cdot \nabla Y
    &&\quad \text{in } L^1((0,T)\times D),\\
    g^\delta Y^\delta &\longto gY
    && \quad \text{in } L^1((0,T)\times D).
  \end{alignat*}
  The proof is completed passing to the limit as $\delta \to 0$ in
  \eqref{eq:delta}, in complete analogy to \cite[\S~4]{cm:inv} and
  \cite[\S~3]{cm:div}.
\end{proof}

\begin{coroll}
  Under the assumptions of the previous proposition, one has
  \[
    Y \in L^0(\Omega; C([0,T];H)).
  \]
\end{coroll}
\begin{proof}
  Since $Y \in L^\infty(0,T; H) \cap C([0,T]; V_0')$, the trajectories
  of $Y$ are weakly continuous in $H$ (see,
  e.g.,\cite{Strauss}). Moreover, by It\^o's formula one has
  \begin{align*}
    &\frac12\norm{Y(t)}^2  - \frac12\norm{Y(r)}^2
      + \int_r^t\!\!\int_D f(s) \cdot \nabla Y(s)\,ds
      + \int_r^t\!\!\int_D g(s)Y(s)\,ds \\
    &\hspace{3em} = 
      \frac12 \int_r^t \norm{G(s)}^2_{\cL^2(U,H)}\,ds
      + \int_r^t Y(s)G(s)\,dW(s)
  \end{align*}
  for every $r,\,t\in[0,T]$. This implies, by an argument analogous to
  the one used in \cite[\S~3]{cm:ref}, that the function
  $t \mapsto \norm{Y(t)}$ is continuous on $[0,T]$. By a well-known
  criterion we thus conclude that $Y$ has strongly continuous
  trajectories in $H$.
\end{proof}

\subsection{Well-posedness in a special case}
As a first step we prove existence of solutions to \eqref{eq:0}
assuming that the noise is of additive type and that
\[
  B \in L^2(\Omega; L^2(0,T; \cL^2(U,V_0))).
\]
For any $\lambda>0$, let $\gamma_\lambda$ and $\beta_\lambda$ denote
the Yosida approximations of $\gamma$ and $\beta$, respectively, and
consider the regularized equation
\[
  du_\lambda(t) - \lambda\Delta u_\lambda(t)\,dt
  - \div\gamma_\lambda(\nabla u_\lambda(t))\,dt
  + \beta_\lambda(u_\lambda(t))\,dt =
  B(t)\,dW(t), \qquad u_\lambda(0)=u_0.
\] 
Since $\gamma_\lambda$ and $\beta_\lambda$ are monotone and
Lipschitz-continuous, it is not difficult to check that the operator
\[
  \phi \longmapsto -\lambda\Delta\phi
  - \div\gamma_\lambda(\nabla\phi) + \beta_\lambda(\phi)
\]
is hemicontinuous, monotone, coercive and bounded on the triple
$(H^1_0(D), H, H^{-1}(D))$, so that the classical results by Pardoux
\cite{Pard} provide existence and uniqueness of a variational solution
\[
  u_\lambda \in L^2(\Omega;C([0,T]; H))
  \cap L^2(\Omega;L^2(0,T;H^1_0(D))).
\]
The a priori estimates on the solution $u_\lambda$ contained in the
next lemma can be obtained essentially as in \cite{cm:div, cm:div2,
  cm:luca, luca}.
\begin{lemma}
  \label{lm:est1}
  There exists a constant $N$ independent of $\lambda$ such that
  \begin{align*}
    &\norm{u_\lambda}^2_{L^2(\Omega; C([0,T]; H))}
      + \lambda\norm{\nabla u_\lambda}^2_{L^2(\Omega; L^2(0,T; H))}\\
    &\hspace{3em} + \norm{\gamma_\lambda(\nabla u_\lambda) \cdot%
      \nabla u_\lambda}_{L^1(\Omega\times(0,T)\times D)} 
      + \norm{\beta_\lambda(u_\lambda)u_\lambda}_{L^1(\Omega\times(0,T)\times D)}
      < N
  \end{align*}
  for all $\lambda \in (0,1)$. Furthermore, there exists
  $\Omega' \in \cF$ with $\P(\Omega')=1$ such that, for every
  $\omega \in \Omega'$, there exists a constant $M(\omega)$
  independent of $\lambda$ such that
   \begin{align*}
     &\norm{u_\lambda(\omega)}^2_{C([0,T]; H)}
       + \lambda\norm{\nabla u_\lambda(\omega)}^2_{L^2(0,T; H)}\\
     &\hspace{3em} + \norm{\gamma_\lambda(\nabla u_\lambda(\omega)) \cdot%
       \nabla u_\lambda(\omega)}_{L^1((0,T)\times D)} 
       + \norm{%
       \beta_\lambda(u_\lambda(\omega))u_\lambda(\omega)}_{L^1((0,T)\times D)}
       < M(\omega)
   \end{align*}
   for all $\lambda \in (0,1)$.
\end{lemma}
\begin{proof}
  It is an immediate consequence of the (proofs of the)
  \cite[Lemmata~4.3--4.7]{cm:div}, for the part involving $\gamma$,
  and \cite[Lemmata~5.3--5.6]{cm:luca}, for the part involving
  $\beta$.
\end{proof}
Since
\begin{align*}
  k^*(\gamma_\lambda(\nabla u_\lambda))
  \leq k^*(\gamma_\lambda(\nabla u_\lambda))
  + k((I+\lambda\gamma)^{-1}\nabla u_\lambda)
  &= \gamma_\lambda(\nabla u_\lambda) \cdot
  (I+\lambda\gamma)^{-1}\nabla u_\lambda\\
  &\leq \gamma_\lambda(\nabla u_\lambda) \cdot \nabla u_\lambda
\end{align*}
and
\[
  j^*(\beta_\lambda(u_\lambda))
  \leq  j^*(\beta_\lambda(u_\lambda)) + j((I+\lambda\beta)^{-1}u_\lambda)
  = \beta_\lambda(u_\lambda)(I+\lambda\beta)^{-1}u_\lambda
  \leq \beta_\lambda(u_\lambda)u_\lambda,
\]
we infer that the families $(k^*(\gamma_\lambda(\nabla u_\lambda)))$
and $(j^*(\beta_\lambda(u_\lambda)))$ are uniformly bounded in
$L^1(\Omega \times (0,T) \times D)$. Therefore, recalling that $k^*$
and $j^*$ are superlinear, thanks to the de la Vall\'ee-Poussin
criterion and the Dunford-Pettis theorem we deduce that the families
$(\gamma_\lambda(u_\lambda))$ and $(\beta_\lambda(u_\lambda))$ are
relatively weakly compact in $L^1(\Omega \times(0,T) \times D)^d$ and
$L^1(\Omega\times(0,T)\times D)$, respectively. Analogously, the
families $(\gamma_\lambda(u_\lambda(\omega)))$ and
$(\beta_\lambda(u_\lambda(\omega)))$ are relatively weakly compact in
$L^1((0,T)\times D)^d$ and $L^1((0,T)\times D)$, respectively, for
$\P$-a.e. $\omega \in \Omega$.

Let $\Omega'$ be as in the previous lemma and take $\omega \in
\Omega'$. Then we have, along a subsequence $\lambda'$ of $\lambda$
depending on $\omega$,
\begin{alignat*}{2}
  u_{\lambda'}(\omega) &\longto u(\omega)
  &&\quad \text{weakly* in } L^\infty(0,T; H),\\
  \nabla u_{\lambda'}(\omega) &\longto \nabla u(\omega)
  &&\quad \text{weakly in } L^1((0,T)\times D)^d,\\
  \lambda' u_{\lambda'}(\omega) &\longto 0
  &&\quad \text{in } L^2(0,T; H^1_0(D)),\\
  \gamma_{\lambda'}(u_{\lambda'}(\omega)) &\longto \eta(\omega)
  &&\quad \text{weakly in } L^1((0,T)\times D)^d,\\
  \beta_{\lambda'}(u_{\lambda'}(\omega)) &\longto \xi(\omega)
  &&\quad \text{weakly in } L^1((0,T)\times D),
\end{alignat*}
hence, by passage to the weak limit in the regularized equation taking
test functions in $V_0$, we have
\begin{equation}
  \label{eq:pt}
  u - \int_0^\cdot \div \eta(s)\,ds + \int_0^\cdot \xi(s)\,ds =
  u_0 + \int_0^\cdot B(s)\,dW(s).
\end{equation}
Moreover, by the lower semicontinuity of convex integrals, it also
follows that
\[
  k(\nabla u(\omega)) + k^*(\eta(\omega))
  + j(u(\omega)) + j^*(\xi(\omega)) \in L^1((0,T)\times D).
\]
Arguing as in \cite[pp.~27--28]{cm:luca} and
\cite[pp.~18--19]{cm:div}, one can show that the process $u$
constructed in this way is unique in the space
$L^2(\Omega; C([0,T]; H))$.  This ensures in turn that the
convergences of $(u_\lambda)$ to $u$ hold along the entire sequence
$\lambda$, which is independent of $\omega$.  In particular, we have
that
\[
  u_\lambda(\omega) \longto u(\omega) \quad \text{ weakly in } L^2(0,T; H)
  \quad\forall\,\omega\in\Omega'.
\]
Since $(u_\lambda)$ is bounded in $L^2(\Omega\times(0,T)\times D)$, we
deduce that $u_\lambda$ converges weakly to $u$ also in
$L^2(\Omega\times(0,T);H)$.  Hence, by a direct application of Mazur's
lemma, we infer that $u$ is a predictable process with values in $H$.
Unfortunately a similar argument does not apply to $\eta$ and
$\xi$. In fact, by uniqueness of $u$, we can only infer from
\eqref{eq:pt} that $-\div \eta + \xi$ is unique: namely, assume
that $(\eta_i(\omega),\xi_i(\omega))$, $i=1,2$, are weak limits in
$L^1(0,T;L^1(D))^{d+1}$ of
$\bigl(\gamma_\lambda(\nabla
u_\lambda(\omega)),\beta_\lambda(u_\lambda)\bigr)$ along two
subsequences of $\lambda$ (depending on $\omega)$. Then
\[
  \int_0^t \bigl( -\div(\eta_1-\eta_2) + (\xi_1-\xi_2) \bigr)\,ds = 0
  \qquad \forall t \in [0,T],
\]
hence $-\div(\eta_1-\eta_2) + (\xi_1-\xi_2) = 0$, or, equivalently,
$-\div\eta_1 + \xi_1 = -\div\eta_2 + \xi_2$ in $V_0'$ for a.a.
$t \in [0,T]$. However, this allows us to claim, setting
$\eta_\lambda := \gamma_\lambda(\nabla u_\lambda)$ and
$\xi_\lambda := \beta_\lambda(u_\lambda)$, that
\[
  -\div \eta_\lambda + \xi_\lambda \longto -\div \eta + \xi
  \quad \text{ weakly in } L^1(0,T;V_0')
  \quad \forall \omega \in \Omega'
\]
along the whole sequence $\lambda$, thanks to the same uniqueness
argument already used for $u$. In fact, let us set, for notational
convenience,
\begin{align*}
  \Phi: L^1(D)^{d+1} &\longto V_0'\\
  (v,f) &\longmapsto -\div v+f
\end{align*}
and $\zeta_\lambda:=(\eta_\lambda,\xi_\lambda)$,
$\zeta:=(\eta,\xi)$. Note that $\Phi$, being a linear bounded
operator, can be extended to a linear bounded operator from
$L^1((0,T) \times D)^{d+1} \simeq L^1(0,T;L^1(D)^{d+1})$ to
$L^1(0,T;V_0')$, also when both spaces are endowed with the weak
topology. Then $\zeta_\lambda \to \zeta$ weakly in
$L^1((0,T) \times D)^{d+1}$ implies that
$\Phi\zeta_\lambda \to \Phi\zeta$ weakly in $L^1(0,T;V_0')$ for all
$\omega \in \Omega'$. Such a convergence, however, does not allow to
infer that $-\div \eta + \xi$ is predictable as a $V_0'$-valued
process. The reason is that we may certainly find, by Mazur's lemma, a
convex combination of $-\div \eta_\lambda + \xi_\lambda$ converging
strongly to $-\div \eta + \xi$ in $L^1(0,T;V_0')$ for all
$\omega \in \Omega'$, but such a convex combination would depend on
$\omega$, bringing us back to the same problem we are trying to
solve.\footnote{We could just say that $-\div\eta + \xi$ is weakly
  measurable with respect to $\mathscr{F}$ and the Borel
  $\sigma$-algebra of $L^1(0,T;V_0')$. Since this space is separable,
  by Pettis' theorem we also have strong measurability. This
  observation, however, does not seem to imply the desired result.} In
order to show that $-\div \eta + \xi$ is indeed predictable, we are
first going to prove that
\[
  -\div\eta_\lambda + \xi_\lambda \longto -\div \eta + \xi
  \quad \text{ weakly in } L^1(\Omega \times (0,T);V_0').
\]
We have just shown that
\[
  \int_0^T \ip[\big]{\Phi\zeta_\lambda(\omega,t)}{\phi(t)} \,dt
  \longto \int_0^T \ip[\big]{\Phi\zeta_\lambda(\omega,t)}{\phi(t)} \,dt
\]
for all $\phi \in L^\infty(0,T;V_0)$, for all $\omega \in \Omega'$,
where $\ip{\cdot}{\cdot}$ stands for the duality between $V_0'$ and
$V_0''=V_0$. Let $\psi \in L^\infty(\Omega \times (0,T);V_0)$. Then
$\psi(\omega,\cdot) \in L^\infty(0,T; V_0)$ for $\P$-a.e.
$\omega \in \Omega$.
Indeed, the set
\[
  A := \bigl\{ (\omega,t) \in \Omega \times [0,T]: \;
  \norm{\psi(\omega,t)}_{V_0} > \norm{\psi}_{L^\infty(\Omega\times(0,T); V_0)}
  \bigr\}
\]
belongs to $\cF \otimes \mathscr{B}([0,T])$, and, by Tonelli's theorem,
\[
  \abs{A} = \int_\Omega \int_0^T 1_A \,dt\,d\P
  = \int_\Omega \operatorname{Leb}(A_\omega)\,\P(d\omega),
\]
where $\abs{A}$ denotes the measure of $A$ and $A_\omega$ stands for
the section of $A$ at $\omega$, i.e.
\[
  A_\omega := \bigl\{ t \in [0,T]: \; (\omega,t) \in A \bigr\},
\]
which belongs to $\mathscr{B}([0,T])$ for $\P$-a.e.
$\omega \in \Omega$.  Since $|A|=0$, it follows that
$\abs{A_\omega}=0$ for $\P$-a.e. $\omega \in \Omega$. This implies, by
definition of $A$, that $\psi(\omega,\cdot) \in L^\infty(0,T)$ for
$\P$-a.e. $\omega \in \Omega$.  Consequently, we have
\[
  \int_0^T \ip[\big]{\Phi\zeta_\lambda(\omega,t)}{\psi(\omega,t)} \,dt
  \longto \int_0^T \ip[\big]{\Phi\zeta(\omega,t)}{\psi(\omega,t)} \,dt
\]
for $\P$-a.e. $\omega \in \Omega$. To complete the argument it is then
enough to show that the left-hand side, as a subset of $L^0(\Omega)$
indexed by $\lambda$, is uniformly integrable. To this end, we collect
some simple facts about uniform integrability in the following lemma.
\begin{lemma}
  Let $(X,\mathscr{A},m)$ be a finite measure space and $I$ an
  arbitrary index set.
  \begin{itemize}
  \item[\emph{(a)}] Let $(f_i)_{i \in I}$,
    $(g_i)_{i \in I} \subset L^0(X;\erre^n)$ be such that
    $\abs{f_i} \leq \abs{g_i}$ for all $i \in I$ and assume that
    $(g_i)$ is uniformly integrable. Then $(f_i)$ is uniformly
    integrable.
  \item[\emph{(b)}] Let $(f_i) \subset L^0(X;\erre^n)$ be uniformly
    integrable and $\phi \in L^\infty(X;\erre^n)$. Then
    $(\phi \cdot f_i) \subset L^0(X)$ is uniformly integrable.
  \item[\emph{(c)}] Let $F: \erre^n \to \erre$ with $F(0)=0$ be convex and
    superlinear at infinity, and
    $(f_i) \subset L^0(X;\erre^n)$ be such that
    $(F \circ f_i)$ is bounded in $L^1(X)$. Then $(f_i)$
    is uniformly integrable.
  \item[\emph{(d)}] Let $(Y,\mathscr{B},n)$ be a further finite measure
    space. If
    $(f_i) \subset L^0(X \times Y,\mathscr{A} \otimes \mathscr{B},m
    \otimes n;\erre^n)$ is uniformly integrable, then
    $(g_i) \subset L^0(X;\erre^n)$ defined by
    \[
      g_i := \int_Y f_i(\cdot,y)\,n(dy)
    \]
    is uniformly integrable.
  \end{itemize}
\end{lemma}
\begin{proof}
  (a) is an immediate consequence of the definition of uniform
  integrability.
  
  \noindent (b) Let $\varepsilon>0$. By assumption, there exists
  $\delta=\delta(\varepsilon)>0$ such that
  \[
    \int_A \abs[\big]{f_i}_{\erre^n}\,dm
    < \frac{\varepsilon}{\norm{\phi}_{L^\infty}}
    \qquad \forall A \in \mathscr{A}, \; m(A)<\delta.
  \]
  Then
  \[
    \int_A \abs[\big]{\phi \cdot f_i}\,dm
    \leq \norm{\phi}_{L^\infty} \int_A
    \abs[\big]{f_i}_{\erre^n}\,dm < \varepsilon.
  \]
  (c) is a variation of the classical criterion by de la
  Vall\'ee-Poussin. A detailed proof (which is nonetheless very close
  to the one in the standard one-dimensional case) can be found in
  \cite{cm:div}.
  
  \noindent (d) Let $\varepsilon>0$. By
  assumption, there exists $\delta'=\delta'(\varepsilon)>0$ such that
  \[
    \int_C \abs[\big]{f_i}_{\erre^n} \,dm \otimes n < \varepsilon \qquad
    \forall C \in \mathscr{A} \otimes \mathscr{B}, \;
    m \otimes n(C) < \delta'.
  \]
  Let $\delta := \delta'/n(Y)$ and $A \in \mathscr{A}$ with
  $m(A)<\delta$. Then
  \[
    \int_A \abs[\bigg]{\int_Y f_i(x,y)\,n(dy)}_{\erre^n}\,m(dx)
    \leq     \int_{A \times Y} \abs[\big]{f_i(x,y)}_{\erre^n}\,m(dx)\,n(dy)
    < \varepsilon
  \]
  because $m \otimes n (A \times Y) = m(A) n(Y) < \delta n(Y) = \delta'$.
\end{proof}
Let us now resume with the main reasoning. Since
\[
  \int_0^T \ip[\big]{\Phi\zeta_\lambda}{\psi} \lesssim
  \norm{\psi}_{L^\infty(\Omega \times (0,T);V_0)} \biggl(
  \int_0^T\!\!\int_D \abs{\eta_\lambda} + \int_0^T\!\!\int_D
  \abs{\xi_\lambda} \biggr),
\]
by parts (a), (b) and (d) of the previous lemma it is sufficient to
show that $(\eta_\lambda)$ and $(\xi_\lambda)$ are
uniformly integrable in $\Omega \times(0,T) \times D$. But this is
true, in view of part (c) of the previous lemma, because
$k^*(\eta_\lambda)$ and $j^*(\xi_\lambda)$ are uniformly bounded in
$L^1(\Omega\times(0,T)\times D)$. Vitali's theorem then yields
\[
  \int_0^T \ip[\big]{\Phi\zeta_\lambda(\omega,t)}{\psi(\omega,t)} \,dt
  \longto \int_0^T \ip[\big]{\Phi\zeta(\omega,t)}{\psi(\omega,t)} \,dt
  \qquad\text{in } L^1(\Omega),
\]
hence, in particular,
\[
  \Phi(\eta_\lambda,\xi_\lambda) \longto \Phi(\eta,\xi)
  \quad \text{ weakly in } L^1(\Omega\times(0,T); V_0').
\]
Furthermore, from the uniform integrability of $(\eta_\lambda)$ and
$(\xi_\lambda)$ in $\Omega\times(0,T)\times D$ it also follows
that, along a subsequence $\mu$ of $\lambda$,
\[
  (\eta_\mu,\xi_\mu) \longto (\bar\eta,\bar\xi)
  \quad \text{ weakly in } L^1(\Omega\times(0,T)\times D)^{d+1},
\]
hence also
\[
  \Phi(\eta_\mu,\xi_\mu)\longto\Phi(\bar\eta, \bar\xi)
  \quad\text{weakly in } L^1(\Omega\times (0,T);V_0').
\]
An application of Mazur's lemma yields, in complete analogy to the
case of $u$, that $\bar{\eta}$ and $\bar{\xi}$ are predictable
processes with values in $L^1(D)^d$ and $L^1(D)$, respectively.  Since
$\mu$ is a subsequence of $\lambda$, by uniqueness of the weak limit
we have that $\Phi(\eta,\xi)=\Phi(\bar\eta, \bar\xi)$, i.e.
\[
  -\div\eta + \xi = -\div\bar{\eta} + \bar{\xi}.
\]
This implies that the identity \eqref{eq:pt} remains true with $\eta$
and $\xi$ replaced by $\bar{\eta}$ and $\bar{\xi}$, respectively. In
other words, modulo relabeling, we can just assume, without loss of
generality, that $\eta$ and $\xi$ in \eqref{eq:pt} are predictable and
that
\[
  (\eta_\lambda, \xi_\lambda) \longto (\eta,\xi) 
  \quad \text{ weakly in } L^1(\Omega\times (0,T)\times D)^{d+1}.
\]
By weak lower semicontinuity and Lemma~\ref{lm:est1}, this also
implies, arguing as in \cite{cm:div, cm:div2, cm:luca, luca}, that
\begin{align*}
  u &\in L^2(\Omega; L^\infty(0,T; H))\cap L^1(\Omega; L^1(0,T; W^{1,1}_0(D))),\\
  \eta &\in L^1(\Omega\times(0,T)\times D)^d,\\
  \xi &\in L^1(\Omega\times(0,T)\times D),\\
  k(\nabla u) + k^*(\eta) = \nabla u \cdot \eta
    &\in L^1(\Omega\times(0,T)\times D),\\
  j(u)+ j^*(\xi) = u \xi &\in L^1(\Omega\times(0,T)\times D).
\end{align*}
In order to show that $\eta \in \gamma(\nabla u)$ and $\xi \in
\beta(u)$ a.e. in $\Omega\times(0,T)\times D$, it suffices to prove,
by the maximal monotonicity of $\gamma$ and $\beta$, that
\begin{equation}
\label{eq:ls}
  \limsup_{\lambda \to 0} \E\int_0^T\!\!\int_D
  \bigl( \eta_\lambda\cdot\nabla u_\lambda + \xi_\lambda u_\lambda \bigr)
  \leq \E\int_0^T\!\!\int_D \bigl(\eta\cdot\nabla u + \xi u \bigr)
\end{equation}
(cf.~\cite[pp.~17-18]{cm:div}). To this purpose, note that the
ordinary It\^o formula and Proposition~\ref{prop:Ito} yield
\[
\frac12 \E\norm{u_\lambda(T)}^2%
+ \E\int_0^T\!\!\int_D%
\bigl(\eta_\lambda \cdot \nabla u_\lambda + \xi_\lambda u_\lambda \bigr) 
= \frac12 \E\norm{u_0}^2
+ \frac12 \E\int_0^T \norm[\big]{B(s)}^2_{\cL^2(U,H)}\,ds
\]
and 
\[
\frac12 \E\norm{u(T)}^2
+ \E\int_0^T\!\!\int_D \bigl( \eta\cdot\nabla u + \xi u \bigr)
= \frac12 \E\norm{u_0}^2 +
\frac12 \E\int_0^T \norm[\big]{B(s)}^2_{\cL^2(U,H)}\,ds,
\]
respectively (the stochastic integrals appearing in both versions of
It\^o's formula are in fact martingales, not just local martingales,
hence their expectation is zero). Since $u_\lambda(T) \to u(T)$ weakly
in $L^2(\Omega;H)$, one has $\E\norm{u(T)}^2 \leq \liminf_{\lambda \to
  0} \E\norm{u_\lambda(T)}^2$, hence, by comparison, \eqref{eq:ls}
follows.

Finally, the strong pathwise continuity (in $H$) of $u$ is an
immediate consequence of the corollary to Proposition~\ref{prop:Ito}.

\begin{rmk}
  Another way to ``restore'' uniqueness of limits for the pair
  $\zeta_\lambda=(\eta_\lambda,\xi_\lambda)$ is to view it as element
  of the quotient space $L^1(D)^{d+1}/M$, where $M:=\ker\Phi$. Note
  that $M$ is a closed subset of $L^1$ (we suppress the superscript as
  well as the indication of the domain just within this remark), as
  the inverse image of the closed set $\{0\}$ through a continuous
  linear map, hence $L^1/M$ is a Banach space. However, working with
  the spaces $L^1(0,T;L^1/M)$ and $L^1(\Omega \times (0,T);L^1/M)$
  present technical difficulties due to the fact that their dual
  spaces are hard to characterize. A bit more precisely, this has to
  do with the fact that the dual of $L^1(0,T;E)$ is $L^\infty(0,T;E')$
  if and only if $E$ has the Radon-Nikodym property. This property is
  enjoyed by reflexive spaces, but not by $L^1$ spaces (see, e.g.,
  \cite{Die:VM}).
\end{rmk}

\subsection{Well-posedness in the general case}
Let us consider now equation \eqref{eq:0} with general additive noise,
i.e. with
\[
  B \in L^2(\Omega; L^2(0,T; \cL^2(U,H))).
\]
Thanks to classical elliptic regularity results, there exists $m \in
\enne$ such that the $(I-\delta\Delta)^{-m}$ is a continuous linear
map from $L^1(D)$ to $W^{1,\infty}(D) \cap H^1_0(D)$ for every
$\delta>0$.  Setting then $V_0:=(I-\Delta)^{-m}(H)$ and
$B^\delta:=(I-\delta\Delta)^{-m}B$, we have $B^\delta \in L^2(\Omega;
L^2(0,T; \cL^2(U,V_0)))$, hence, by the well-posedness results already
obtained, the equation
\[
du^\delta - \div\gamma(\nabla u^\delta)\,dt + \beta(u^\delta)\,dt \ni
B^\delta\,dW, \qquad u^\delta(0) = u_0,
\]
admits a strong solution $(u^\delta, \eta^\delta, \xi^\delta)$.
Arguing as in \cite{cm:div, cm:div2, cm:luca, luca}, one can show
using It\^o's formula that $(u^\delta)$ is a Cauchy sequence in
$L^2(\Omega; C([0,T]; H))$ and that $(\nabla u^\delta)$,
$(\eta^\delta)$, and $(\xi^\delta)$ are relatively weakly compact in
$L^1(\Omega \times(0,T) \times D)$, so that
\begin{alignat*}{2}
  u^\delta &\longto u &&\quad \text{ in } L^2(\Omega; C([0,T]; H)),\\
  u^\delta &\longto u &&\quad \text{ weakly in } 
  L^1(\Omega\times(0,T); W^{1,1}_0(D)),\\
  \eta^\delta &\longto \eta &&\quad\text{ weakly in }
  L^1(\Omega\times(0,T)\times D)^d,\\
  \xi^\delta &\longto \xi &&\quad \text{ weakly in }
  L^1(\Omega\times(0,T)\times D),
\end{alignat*}
from which it follows that $(u,\eta,\xi)$ solves the original
equation.  Moreover, the strong-weak closure of $\beta$ readily
implies that $\xi \in \beta(u)$ a.e. in $\Omega \times (0,T) \times
D$. Finally, arguing as in the previous subsection, by weak lower
semicontinuity of convex integrals and It\^o's formula one can show
that
\[
\limsup_{\lambda \to 0} \E\int_0^T\!\!\int_D \eta_\lambda \cdot 
\nabla u_\lambda 
\leq \E\int_0^T\!\!\int_D \eta \cdot \nabla u,
\]
so that $\eta \in \gamma(\nabla u)$ a.e. in $\Omega \times (0,T)
\times D$ as well.

Continuous dependence on the initial datum is a consequence of It\^o's
formula and the monotonicity of $\gamma$ and $\beta$.  Finally, the
generalization to the case of multiplicative noise follows using the
Lipschitz continuity of $B$ and a classical fixed point argument. A
detailed exposition of the arguments needed to prove these claims can
be found in \cite{cm:div, cm:div2, cm:luca, luca}.

\bibliographystyle{amsplain}
\bibliography{ref}

\end{document}